%% file: edgepolytope.tex
\documentclass[12pt]{amsart}
\usepackage[dvips]{graphicx}
\usepackage{amssymb,amsmath,amsthm}

\def\NZQ{\mathbb}    % the font for N,Z,Q,R,C

\def\RR{{\NZQ R}}

\def\frk{\mathfrak}    % font for "Fraktur"

\def\Phi{{\frk N}}

\def\eb{{\bf e}}

\newtheorem{Theorem}{Theorem}[section]
\newtheorem{Lemma}[Theorem]{Lemma}

\newtheorem{Proposition}[Theorem]{Proposition}

\theoremstyle{definition}
\newtheorem{Remark}[Theorem]{Remark}

\newtheorem{Example}[Theorem]{Example}

\newtheorem{Conjecture}[Theorem]{Conjecture}

\begin{document}

\title{The number of edges of the edge polytope of \\ a finite simple graph}
\author{Takayuki Hibi, Aki Mori, Hidefumi Ohsugi and Akihiro Shikama}

% \thanks{ }
\address{Takayuki Hibi,
Department of Pure and Applied Mathematics,
Graduate School of Information Science and Technology,
Osaka University,
Toyonaka, Osaka 560-0043, Japan}
\email{hibi@math.sci.osaka-u.ac.jp}

\address{Aki Mori,
Department of Pure and Applied Mathematics,
Graduate School of Information Science and Technology,
Osaka University,
Toyonaka, Osaka 560-0043, Japan}
\email{a-mori@cr.math.sci.osaka-u.ac.jp}

\address{Hidefumi Ohsugi,
Department of Mathematical Sciences,
School of Science and Technology,
Kwansei Gakuin University,
Sanda, Hyogo, 669-1337, Japan}
\email{ohsugi@kwansei.ac.jp}

\address{Akihiro Shikama,
Department of Pure and Applied Mathematics,
Graduate School of Information Science and Technology,
Osaka University,
Toyonaka, Osaka 560-0043, Japan}
\email{a-shikama@cr.math.sci.osaka-u.ac.jp}


\begin{abstract}
Let $d \geq 3$ be an integer.
It is known that the number of edges of 
the edge polytope of the complete graph with $d$ vertices
is $d(d-1)(d-2)/2$.
In this paper, we study the maximum possible number $\mu_d$
of edges of
the edge polytope arising from 
finite simple graphs with $d$ vertices.
We show that $\mu_{d}=d(d-1)(d-2)/2$
if and only if $3 \leq d \leq 14$.
In addition, we study the asymptotic behavior of $\mu_d$.
Tran--Ziegler gave a lower bound for $\mu_d$ by 
constructing a random graph.
We succeeded in improving this bound by 
constructing both a non-random graph
and a random graph whose complement is bipartite.
\end{abstract}

\maketitle

\vspace{-0.5cm}

\input{introduction}
\input{Section1new}
\input{Section2new}

\end{document}

%% file: introduction.tex
\section*{Introduction}
The number of $i$-dimensional faces of a convex polytope
has been studied by many researchers for a long time.
One of the most famous classical results is ``Euler's formula."
Extremal problem concerning number of faces
is an important topic in the study of convex polytopes.
On the other hand, 
the study of edge polytopes of finite graphs 
has been conducted by many authors
from viewpoints of commutative algebra on toric ideals and combinatorics
of convex polytopes. 
We refer the reader to \cite{OhHinormal, OhHiquadratic}
for foundations of edge polytopes.  
Faces of edge polytopes are studied in, e.g., 
\cite{OhHinormal, OhHimulti, OH}.
Recently, Tran and Ziegler \cite{Ziegler}  studied extremal problem
on edge polytopes.
In particular, using \cite[Lemma 1.4]{OH},
they gave
bounds for the maximum possible number $\mu_d$
of edges of
the edge polytope arising from 
finite simple graphs with $d$ vertices.
Following \cite[Question 1.3]{HLZ}, 
we wish to find a finite simple graph
$G$ with $d$ vertices such that the edge polytope
of $G$ has  $\mu_{d}$ edges
and to compute $\mu_{d}$.

Recall that a finite   
{\em{simple}} graph is a finite graph with no loops and no multiple edges.
Let $[d] = \{ 1, \ldots, d \}$ be the vertex set
and $\Omega_{d}$ the set of  
finite simple graphs on $[d]$, where $d \geq 3$.
Let ${\eb}_i$ denote the $i$th unit coordinate
vector of the Euclidean space ${\RR}^d$. 
Let $G \in \Omega_{d}$ and $E(G)$ the set of edges of $G$.
If $e = \{i,j\} \in E(G)$, then we set
$\rho(e) = {\eb}_i + {\eb}_j \in {\RR}^d$.
The {\em edge polytope} ${\mathcal P}_G$ of $G \in \Omega_{d}$ 
is the convex hull of the finite set $\{ \rho(e) \, : \, e \in E(G) \}$
in $\RR ^d$. 
Let $\varepsilon(G)$ denote the number of edges, namely 1-dimensional faces, of 
$\mathcal{P}_G$. 
For example, consider the case of the complete graph $K_d$ on $[d]$.
By \cite[Lemma 1.4]{OH},
for edges $e$ and $f$ $(e \neq f)$ of  $K_d$, 
the convex hull of $\{ \rho(e) ,\rho(f) \}$ is an edge of 
the edge polytope $\mathcal{P}_{K_d}$ if and only if 
$e$ and $f$ have a common vertex.
Hence, 
$\varepsilon(K_d) = d {d-1 \choose 2}=d(d-1)(d-2)/2$.
On the other hand,
$
\varepsilon(K_{m,n}) = \ mn(m+n-2)/2,
$
where $K_{m,n}$ is the complete bipartite graph on the vertex set
$[\,m\,]\cup \{ m+1,\ldots, m+n\}$
for which $m , n \geq 1$
% and $m + n \geq 3$
(see \cite[Theorem 2.5]{OhHimulti}).
In this paper, we are interested in %the number
$
\mu_{d} = 
\max \{ \, \varepsilon(G) \, : \, G \in \Omega_{d} \, \}
$ for $d \geq 3$.

\begin{Theorem}
\label{theorem}
%Work with the same notation as above.
For an integer $d \geq 3$, let $\Omega_d$ be the set of 
finite simple graphs on $[d]$.
Given a graph $G \in \Omega_d$, 
let $\varepsilon(G)$ denote the number of 
edges of the edge polytope $\mathcal{P}_G$ of $G$.
Then, the following holds:
\begin{itemize}
\item[(a)] 
If $3 \leq d \leq 13$ and $G \in \Omega_{d}$ with $G \neq K_d$, then
$\varepsilon({G})< \varepsilon({K_d})$. 
\item[(b)] 
Let $G \in \Omega_{14}$ with $G \neq K_{14}$.
Then
$\varepsilon({G}) \leq \varepsilon( K_{14})$.
Moreover, $\varepsilon({G}) = \varepsilon({K_{14}})$
if and only if either $G=K_{14} - K_{4,5}$ or
$G=K_{14}-K_{5,5}$.
\item[(c)] 
If $d \geq 15$, then there exists $G \in \Omega_{d}$
such that $\varepsilon(G) >
\varepsilon(K_d) $.
%= 50(d - 14)$. 
\end{itemize} 
\end{Theorem}

We devote Section $1$ to giving a proof of Theorem \ref{theorem}.  
At present, for $d \geq 15$, it remains unsolved to find $G \in \Omega_{d}$ with
$\mu_{d} = \varepsilon(G)$ and to compute $\mu_{d}$.
(Later, we will see that $\mu_{15} \geq \varepsilon(K_{15}) + 50=
1415$.)
In Section 2, we study the asymptotic behavior of $\mu_d$.
Recently, Tran--Ziegler \cite{Ziegler} 
gave a lower bound for $\mu_d$ by a random graph:
$$
\varepsilon (G(d, 1/\sqrt{3})) = \frac{1}{54} d^4 
+ \frac{1}{18} d^3
- \frac{8 }{27} d^2
+ \frac{1}{3} d.
$$
They also gave an upper bound for $\mu_d$:
$\mu_d \leq (\frac{1}{32} + o(1)) d^4$.
(However, this upper bound is not sharp.
See \cite[Remark]{Ziegler}.)
In this paper, we succeeded in improving their lower bound by 
constructing a non-random graph (see Example \ref{twocompletebipartite})
and a random graph whose complement is bipartite
(see Theorem \ref{randombipartite}):
$$
\varepsilon ({\mathbb G}) = \frac{5\sqrt{5}-11}{8} \  d^4 
- \frac{12 \sqrt{5} - 27}{2} d^3
+\frac{19 \sqrt{5} - 44}{2} d^2
+d,
$$
where ${\mathbb G} = K_d  - G(K_{d/2,d/2},p)$ with $p=3-\sqrt{5}$.
These results suggest the following:% conjecture:
\begin{Conjecture}
Let $G \in \Omega_{d}$ with $\mu_{d} = \varepsilon(G)$.
Then, the complement of $G$ is a bipartite graph.
\end{Conjecture}
Note that, by Theorem \ref{theorem}, this conjecture is true for $3 \leq d \leq 14$.

%% file: Section1new.tex
\section{Proof of Theorem \ref{theorem}}
In this section, we give a proof of Theorem \ref{theorem}.
The following lemma is studied in \cite[Lemma 1.4]{OH}.

\begin{Lemma}\label{cc}
Let $e$ and $f$ $(e \neq f)$ be edges of a graph $G \in \Omega_d$. 
Then, the convex hull of $\{ \rho(e) ,\rho(f) \}$ is an edge of 
the edge polytope $\mathcal{P}_G$ if and only if one of the following conditions is satisfied.
\begin{itemize}
\item[(i)] $e$ and $f$ have a common vertex in $[d]$.
\item[(ii)] $e = \{ i,j \}$ and $f = \{k , l \}$ have no common 
vertices, and
the induced subgraph of $G$ on the vertex set $\{ i,j,k,l \}$ 
 has no cycles of length $4$.
\end{itemize}
\end{Lemma}

The {\em complementary} graph $\overline{G}$ of 
a graph $G \in \Omega_d$
is the graph whose vertex set is $[d]$
and whose edges are the non-edges of $G$.
For a vertex $i$ of a graph $G$,
let $\deg_G(i)$ denote the degree of $i$ in $G$.
We translate Lemma \ref{cc} in terms of the complement $\overline{G}$ of $G$.

\begin{Lemma}\label{SHIKAMA}
Let $H$ be the complement of a graph $G \in \Omega_d$
Then, we have
\begin{eqnarray*}
\varepsilon(G) &=& \sum_{i=1}^d \binom{d-1 - \deg_H(i)}{2} 
+a(H) + b(H) + c(H)\\
&=&
\varepsilon(K_d) +
\frac{1}{2}
\sum_{i=1}^d \deg^2_H(i)
-(2d-3) |E(H)|
+a(H) + b(H) + c(H),
\end{eqnarray*}
where $a(H)$, $b(H)$ and $c(H)$ are
the number of induced subgraphs of $H$
on $4$ vertices
of the form
%\begin{enumerate}
%\item[(a)]
(a) a path of length $3$;
%\item[(b)]
(b) a cycle of length $4$;
%\item[(c)]
(c) a path of length $2$ and one isolated vertex,
%\end{enumerate}
respectively.
\end{Lemma}

\begin{proof}
First, the number of pairs of edges satisfying Lemma \ref{cc} (i) is equal to
\begin{eqnarray*}
\sum_{i=1}^d \binom{d-1 - \deg_H(i)}{2} 
&=&
%\sum_{i=1}^d \frac{(d-1 - \deg_H(i))(d-2 - \deg_H(i))}{2} \\
%&=&
\sum_{i=1}^d \frac{(d-1)(d-2) - (2d-3)\deg_H(i) + \deg^2_H(i)}{2} \\
&=&
\varepsilon(K_d) +
\frac{1}{2}
\sum_{i=1}^d \deg^2_H(i)
-(2d-3) |E(H)|.
\end{eqnarray*}
Second, the number of pairs of edges satisfying Lemma \ref{cc} (ii) is 
equal to the number of 
the induced subgraphs $W$ of $G$
% of the form (a), (b) or (c) in the below.
%satisfying
where $W$ is one of the following:
%\begin{enumerate}
%\item[(a')]
(a')
$W$ is a path of length 3;
%\item[(b')]
(b')
$W$ consists of two disjoint edges;
%\item[(c')]
(c')
$W$ is a graph on $\{i,j,k,\ell\}$
with
$E(W)=
\{
\{i,j\}, \{j,k\}, \{i,k\}, \{k,\ell\}
\}$.
%\end{enumerate}
Note that each induced subgraph has exactly one such pair of edges.
The complement of each (a'), (b'), and (c') is
%\begin{enumerate}
%\item[(a)]
%a path of length $3$;
%\item[(b)]
%a cycle of length $4$;
%\item[(c)]
%a path of length $2$ and one isolated vertex,
%\end{enumerate}
(a), (b) and (c), respectively.
\end{proof}

For a graph $H \in \Omega_r$ with $r \leq d$,
let $K_d - H$ denote the graph $G \in \Omega_d$ such that
$E(G) = E(K_d) \setminus E(H)$. 
Using Lemma \ref{SHIKAMA}, we have the following:

\begin{Proposition}
\label{ShikamaConjecture}
Let $H \in \Omega_r$ and let $\psi(H)$ denote 
the number of induced paths in $H$ of length $2$.
Then, the function
$
\varphi (d) = \varepsilon (K_d-H) - \varepsilon (K_d)
$
for 
$
d = r, r+1,r+2,\ldots
$
is a linear polynomial of $d$
whose leading coefficient is $\psi (H) -2 |E(H)|$.
\end{Proposition}

\begin{proof}
Since $d$ is a natural number it is sufficient to show that 
$\varphi (d+1) - \varphi (d) = \psi (H) -2 |E(H)|$ for any $d$.
Let $H_1 = \overline{K_d-H}$ and $H_2 =\overline{K_{d+1}-H}$.
Then, $H_2$ is obtained by adding one isolated vertex
$d+1$ to $H_1$.
Hence, it follows that
$a(H_1) = a(H_2)$, 
$b(H_1) = b(H_2)$,
$c(H_1) + \psi(H) = c(H_2)$
and $\deg_{H_1} (i) = \deg_{H_2}(i)$ for all $1 \leq i \leq d$.
Thus, by Lemma \ref{SHIKAMA},
we have
\begin{eqnarray*}
& &
\varphi (d+1) - \varphi (d)\\
 &=& 
\varepsilon (K_{d+1}-H) - \varepsilon (K_{d+1})
-
\varepsilon (K_d-H) + \varepsilon (K_d)\\
\
 &=& 
\sum_{i=1}^{d+1} {d - \deg_{H_2}(i) \choose 2} 
-
\sum_{i=1}^d {d-1 - \deg_{H_1}(i) \choose 2} 
+\psi(H) \\
& &
+\frac{d(d-1)(d-2)}{2}
-\frac{(d+1)d(d-1)}{2}\\
 &=& 
{d \choose 2}
+
\sum_{i=1}^{d}
\left( {d - \deg_{H_1}(i) \choose 2} 
-
{d-1 - \deg_{H_1}(i) \choose 2} 
\right)
+\psi(H) 
-\frac{3d(d-1)}{2}\\
 &=& 
{d \choose 2}
+
\sum_{i=1}^{d}
(d-1 - \deg_{H_1}(i))
+\psi(H) 
-\frac{3d(d-1)}{2}\\
&=&
\psi(H) 
-\sum_{i=1}^{d}
\deg_{H_1}(i)\\
&=&
\psi (H) -2 |E(H)|,
\end{eqnarray*}
as desired.
\end{proof}

\begin{Proposition}\label{ConnectedComponent}
Let $G \in \Omega_{d}$ and let
$H_1, H_2, \ldots ,H_m$ be all the nonempty connected components of $\overline{G}$.
Then,
$\varepsilon(K_d) - \varepsilon(G) 
= \sum_{j=1}^m ( \varepsilon(K_d) - \varepsilon(K_d - H_j) ).$
 \end{Proposition}

\begin{proof}
Let $H = \overline{G}$ and 
let $H_j' = \overline{K_d - H_j}$
for $1 \leq j \leq m$.
Then, it is easy to see that
$
|E(H)| = \sum_{j =1}^m |E(H_j')|
$,
$
\sum_{i=1}^d \deg^2_H(i)
= 
\sum_{j=1}^m
\sum_{i =1}^d \deg^2_{H_j'}(i)
$,
$
a(H)= \sum_{j =1}^m a(H_j')
$,
$
b(H)= \sum_{j =1}^m b(H_j')
$, and
$
c(H)= \sum_{j =1}^m c(H_j')
$.
Thus, by Lemma \ref{SHIKAMA}, we are done.
\end{proof}

A graph $G \in \Omega_d$ is called {\em bipartite}
if $[d]$ admits a partition into two sets of vertices $V_1$
and $V_2$ such that,
for every edge $\{i,j\}$ of $G$,
either $i \in V_1, j \in V_2$ or $j \in V_1, i \in V_2$ is satisfied.
A {\em complete bipartite} graph is a bipartite graph such that every 
pair of vertices $i,j$ with $i \in V_1$ and $j \in V_2$ is adjacent.
Let $K_{m,n}$ denote the complete bipartite graph with
$|V_1| = m$ and $|V_2|=n$.

\begin{Proposition}\label{completebipartite}
Let $G = K_d - K_{m,n}$ such that
$m+n \le d$ and $m,n \ge 1$. Then,
 \[
\varepsilon(G) -
\varepsilon(K_d)
=
\frac{1}{2}mn (m+n-6)
d
-
\frac{1}{4} mn(3mn+2m^2+2n^2-5m-5n-13).\]
\end{Proposition}

\begin{proof}
Let $H= K_{m,n}$.
Then,
$$
\psi (H) - 2 |E(H)| =  m {n \choose 2}+n {m \choose 2} 
-2 mn =
\frac{1}{2}mn (m+n-6).
$$
Moreover, since $K_{m+n} - K_{m,n}$ is
the disjoint union of $K_m$ and $K_n$, we have
\begin{eqnarray*}
\varphi (m+n) &=& 
%\left(
\frac{m(m-1)(m-2)}{2}
+
\frac{n(n-1)(n-2)}{2}
+
{m \choose 2} {n \choose 2}
%\right)
\\
 & &-\frac{(m+n)(m+n-1)(m+n-2)}{2}\\
&=&
\frac{1}{4} mn(mn-7m-7n+13)
\end{eqnarray*}
by Lemma \ref{cc}.
Hence, by Proposition \ref{ShikamaConjecture},
\begin{eqnarray*}
\varepsilon (G)-\varepsilon (K_d) 
&=&
\frac{1}{2}mn (m+n-6)
(d-(m+n))
+\frac{1}{4} mn(mn-7m-7n+13)\\
&=&
\frac{1}{2}mn (m+n-6)
d
-
\frac{1}{4} mn(3mn+2m^2+2n^2-5m-5n-13),
\end{eqnarray*}
as desired.
 \end{proof}

%Next, we prove that the complete graph $K_d$ with 
%$d \leq 14$ vertices maximize $\varepsilon(G)$.
Let $k_3(H)$ denote the number of triangles
(i.e., cycles of length $3$) of $H$.
The following lemma is important.

\begin{Lemma}\label{keylemma}
Let $H$ be the complement graph of $G \in \Omega_d$.
Then, we have
$$\varepsilon(G) 
\leq
\varepsilon(K_d)
+
\frac{d^2-16d+29}{7} |E(H)|
-
\frac{3}{7} (d -8)
k_3(H).
$$
\end{Lemma}

\begin{proof}
The number of pairs of edges satisfying Lemma 1.1 (i) is,
by Lemma \ref{SHIKAMA}, 
$
\varepsilon(K_d)
-(2d-3) |E(H)|
+
\frac{1}{2}
\sum_{i=1}^d \deg^2_H(i).
$
For an edge $\{i,j\}$ of $H$,
let $k_3(i,j)$ be the number of triangles in $H$ containing $\{i,j\}$.
We define three subsets of $[d] \setminus \{i,j\}$:
\begin{eqnarray*}
X_{i,j} &=& \{ \ell \in   [d] \setminus \{i,j\} : \{i,\ell\} \in E(H), \{j,\ell\} \notin E(H)    \},\\
Y_{i,j} &=& \{ \ell \in   [d] \setminus \{i,j\} : \{j,\ell\} \in E(H), \{i,\ell\} \notin E(H)    \},\\
Z_{i,j} &=& \{ \ell \in   [d] \setminus \{i,j\} : \{i,\ell\} \notin E(H), \{j,\ell\} \notin E(H)    \}.
\end{eqnarray*}
It then follows that, $|X_{i,j}| + |Y_{i,j}| + |Z_{i,j}| + k_3(i,j) = d-2$, and
\begin{eqnarray*}
\frac{1}{2}
\sum_{i=1}^d \deg^2_H(i)
&=&
\frac{1}{2}
\sum_{\{i,j\} \in E(H)} 
(\deg_H(i)+\deg_H(j))\\
&=&
\frac{1}{2}
\sum_{\{i,j\} \in E(H)} 
(|X_{i,j}|+|Y_{i,j}|+2 k_3(i,j)+2)\\
&=&
|E(H)|
+
3 k_3(H)
+
\frac{1}{2}
\sum_{\{i,j\} \in E(H)} 
(|X_{i,j}|+|Y_{i,j}|).
\end{eqnarray*}

Second, we count the number of pairs satisfying Lemma 1.1 (ii).
By Lemma \ref{SHIKAMA}, this number is equal to 
$a(H) + b(H) + c(H)$.
% where $a(H)$, $b(H)$ and $c(H)$ are
%the number of induced subgraphs of $H$
%on $4$ vertices
%of the form
%\begin{enumerate}
%\item[(a)]
%a path of length $3$;
%\item[(b)]
%a cycle of length $4$;
%\item[(c)]
%a path of length $2$ and one isolated vertex,
%\end{enumerate}
%respectively.
Here, we count the number of the induced subgraphs $H'$
of type (a), (b) and (c)
containing an edge $e=\{i,j\}$ of $H$.
If $e$ is an edge of $H'$,
%contained in the above induced subgraph, then
then the other two vertices $\ell$ and $m$ of 
%the induced subgraph
$H'$ satisfy exactly one of
the following conditions:
\begin{enumerate}
\item[(i)]
$\ell \in X_{i,j}, m \in Y_{i,j}$;
\item[(ii)]
$\ell \in Y_{i,j}, m \in Z_{i,j}$;
\item[(iii)]
$\ell \in Z_{i,j}, m \in X_{i,j}$.
\end{enumerate}
If $i,j, \ell, m$ satisfy condition (i), 
then one of the following holds:
\begin{itemize}
\item
$H'$ is a path $(e_1,e_2,e_3)$
and $e = e_2$ (type (a)) ;
\item
$H'$ is a cycle of length 4 
and $e$ is one of four edges (type (b)).
\end{itemize}
It then follows that
$$
%\begin{eqnarray}
a(H) + 4 b(H) =
\sum_{\{i,j\} \in E(H)} 
|X_{i,j}||Y_{i,j}| 
.
%\label{weight}
%\end{eqnarray}
$$
If $i,j, \ell, m$ satisfy either condition (ii) or (iii),
then one of the following holds:
\begin{itemize}
\item
$H'$ is a path $(e_1,e_2,e_3)$
and $e \in \{e_1, e_3\}$ (type (a)) ;
\item
$H'$ is a path $(e_1,e_2)$ with one isolated vertex
and $e \in \{e_1,e_2\}$ (type (c)).
\end{itemize}
It then follows that
$$
%\begin{eqnarray}
2a(H) + 2 c(H)=
\sum_{\{i,j\} \in E(H)} 
\left(
|Y_{i,j}||Z_{i,j}| + 
|Z_{i,j}||X_{i,j}| 
\right).
%\label{weight}
%\end{eqnarray}
$$
%We need to care that by counting among all edges $e=\{i,j\}$ of $H$,
%an induced subgraph of type (a) appears
%once as the form (i), twice as the form (ii) or (iii),
%an induced subgraph of type (b) appears four times as the form (i),
%an induced subgraph of type (c) appears twice as the form (ii) or (iii).
%Thus, we set the weight of the form (i), (ii), (iii) for $1/4$, $1/2$, $1/2$.
Thus, we have
%the total number of induced subgraphs of $H$ in the above statement is
%at most
$$
%\begin{eqnarray}
a(H) + b(H) + c(H)=
-\frac{a(H)}{4}
+\! \! \!
\sum_{\{i,j\} \in E(H)} 
\left(
\frac{1}{4}
|X_{i,j}||Y_{i,j}| + 
\frac{1}{2}
|Y_{i,j}||Z_{i,j}| + 
\frac{1}{2}
|Z_{i,j}||X_{i,j}| 
\right).
%\label{weight}
%\end{eqnarray}
$$
Subject to $|X_{i,j}| + |Y_{i,j}| + |Z_{i,j}| = d-2-k_3(i,j)$, 
we study an upper bound of
$$
\alpha=
\sum_{\{i,j\} \in E(H)} 
\left(
\frac{
|X_{i,j}|+|Y_{i,j}| 
}{2}
+
\frac{1}{4}
|X_{i,j}||Y_{i,j}| + 
\frac{1}{2}
|Y_{i,j}||Z_{i,j}| + 
\frac{1}{2}
|Z_{i,j}||X_{i,j}|
\right).
$$
Each summand of $\alpha$ satisfies
\begin{eqnarray*}
& &
\frac{
|X_{i,j}|+|Y_{i,j}| 
}{2}
+
\frac{1}{4}
|X_{i,j}||Y_{i,j}| + 
\frac{1}{2}
|Y_{i,j}||Z_{i,j}| + 
\frac{1}{2}
|Z_{i,j}||X_{i,j}|\\
&=&
%\frac{1}{4}
%|X_{i,j}||Y_{i,j}| + 
%\frac{1}{2}
%(|X_{i,j}|+|Y_{i,j}| )
%(|Z_{i,j}|+ 1 )\\
%&=&
\frac{1}{4}
|X_{i,j}||Y_{i,j}| + 
\frac{1}{2}
(|X_{i,j}|+|Y_{i,j}| )
(d-1 -k_3(i,j) -( |X_{i,j}|+ |Y_{i,j}| ))\\
&\leq &
\frac{1}{4}
\left(
\frac{
|X_{i,j}|+|Y_{i,j}| 
}
{2}
\right)^2
+ 
\frac{1}{2}
(|X_{i,j}|+|Y_{i,j}| )
(d-1 -k_3(i,j) -( |X_{i,j}|+ |Y_{i,j}| ))\\
& = &
-
\frac{7}{16}
(|X_{i,j}|+|Y_{i,j}|)^2
+ 
\frac{d-1 -k_3(i,j)}{2}
(|X_{i,j}|+|Y_{i,j}| ).
\end{eqnarray*}
The last function has the maximum number
$
\frac{1}{7}
(d-1 -k_3(i,j))^2
$
 when
$
|X_{i,j}|+|Y_{i,j}| = \frac{4}{7} (d-1 -k_3(i,j))
$.
Hence, 
\begin{eqnarray*}
\sum_{\{i,j\} \in E(H)} 
\frac{1}{7}
(d-1 -k_3(i,j))^2
&\leq &
\sum_{\{i,j\} \in E(H)} 
\frac{1}{7}
(d-1)
(d-1 -k_3(i,j))\\
&=&
\frac{1}{7}
\sum_{\{i,j\} \in E(H)} 
(d-1)^2
-
\frac{1}{7}
\sum_{\{i,j\} \in E(H)} 
(d-1)k_3(i,j)\\
&=&
\frac{1}{7}
(d-1)^2
|E(H)|
-
\frac{3}{7}
(d-1)
k_3(H)
\end{eqnarray*}
is an upper bound of $\alpha$.
Thus,
$$
\varepsilon(K_d)
-(2d-3) |E(H)|
+
|E(H)|
+
3 k_3(H)
+
\frac{1}{7}
(d-1)^2
|E(H)|
-
\frac{3}{7}
(d-1)
k_3(H)
$$
is an upper bound of $\varepsilon(G)$ as desired.
\end{proof}

%\subsection{Proof of Theorem 0.1}

Using Proposition \ref{completebipartite}
and Lemma \ref{keylemma}, we prove
Theorem \ref{theorem}.

\begin{proof}[Proof of Theorem 0.1]
(a)
Let $3 \leq d \leq 13$ and $G \in \Omega_{d}$ with $G \neq K_d$.
If $d=3$, then $\varepsilon({G})< \varepsilon({K_d})$ is trivial.
If $d=4$, then $ \varepsilon({K_4}) = 12$.
Since $|E(G)| < 6$,
%the number of edges of $G$ is less than or equal to $5$,
we have $ \varepsilon(G) \leq {5 \choose 2} =10 < \varepsilon({K_4}) $.
Let $d \geq 5$ and let $H$ be the complement graph of $G$.
By Lemma \ref{keylemma},
$$
\varepsilon({G})- \varepsilon({K_d})
\leq
\frac{d^2-16d+29}{7} |E(H)|
-
\frac{3}{7} (d -8)
k_3(H).
$$
If $8 \leq d \leq 13$, then 
$
\varepsilon({G})- \varepsilon({K_d})
%\leq
%\frac{d^2-16d+29}{7} |E(H)|
%-
%\frac{3}{7} (d -8)
%k_3(H) 
< 0
$
since
$\frac{d^2-16d+29}{7} < 0$, $|E(H)| > 0$ and $k_3(H) \geq 0$.
Let $5 \leq d \leq 7$.
Then,
$$
\varepsilon({G})- \varepsilon({K_d})
\leq 
\left\{
\begin{array}{cc}
-
\frac{26}{7} |E(H)|
+
\frac{9}{7}
k_3(H) & \mbox{if } d =5,
\medskip
\\
-
\frac{31}{7} |E(H)|
+
\frac{6}{7}
k_3(H)  & \mbox{if } d =6,
\medskip
\\
-
\frac{34}{7} |E(H)|
+
\frac{3}{7}
k_3(H) & \mbox{if } d =7.
\end{array}
\right.
$$
Hence,
if $k_3(H) \leq 2$, then $\varepsilon({G})- \varepsilon({K_d})$ is negative.
On the other hand, if $k_3(H) \geq 3$,
then $|E(H)| \geq 5$.
Since $k_3 (H) \leq {d \choose 3}$, it follows that
$\varepsilon({G})- \varepsilon({K_d})$ is negative.

(b)
Let $G \in \Omega_{14}$ with $G \neq K_{14}$
and let $H=\overline{G}$.
%By Lemma \ref{keylemma}, we have
%$$
%\varepsilon({G})- \varepsilon({K_d})
%\leq
%\frac{1}{7} |E(H)|
%-
%\frac{18}{7}
%k_3(H).
%$$
%However, this is not enough.
We need to evaluate the function which appears in the proof of
Lemma \ref{keylemma} more accurately by focusing on $d = 14$.
Let $|Z_{i,j}| = 12 - k_3(i,j) - |X_{i,j}|- |Y_{i,j}|$ and 
\begin{eqnarray*}
f& =&
\frac{
|X_{i,j}|+|Y_{i,j}| 
}{2}
+
\frac{1}{4}
|X_{i,j}||Y_{i,j}| + 
\frac{1}{2}
|Y_{i,j}||Z_{i,j}| + 
\frac{1}{2}
|Z_{i,j}||X_{i,j}| \\
g&=&
-
\frac{7}{16}
(|X_{i,j}|+|Y_{i,j}|)^2
+ 
\frac{13 -k_3(i,j)}{2}
(|X_{i,j}|+|Y_{i,j}| )
\end{eqnarray*}
be functions of $|X_{i,j}|$ and $|Y_{i,j}|$.
Recall that $f \leq g \leq \frac{1}{7} (13 -k_3(i,j))^2$
and
$
g=\frac{1}{7} (13 -k_3(i,j))^2$
when
$
|X_{i,j}|+|Y_{i,j}| = \frac{4}{7} (13 -k_3(i,j))
$.
If
$
1 \leq k_3(i,j) \leq 12
$, then 
%the maximum value of (\ref{upperfunction}) satisfies
$$
\frac{1}{7}
(13 -k_3(i,j))^2
=
24 - \frac{13}{7}k_3(i,j)
-
\frac{11}{7}
+
\frac{1}{7}
(k_3(i,j)-1)(k_3(i,j)-12)
<
24 - \frac{13}{7}k_3(i,j).
$$
If $k_3(i,j)=0$, then
%the maximum value of (\ref{upperfunction}) is 
$\frac{1}{7} (13 -k_3(i,j))^2 = 24+1/7$.
However, since
$$
4 \left(
\frac{
|X_{i,j}|+|Y_{i,j}| 
}{2}
+
\frac{1}{4}
|X_{i,j}||Y_{i,j}| + 
\frac{1}{2}
|Y_{i,j}||Z_{i,j}| + 
\frac{1}{2}
|Z_{i,j}||X_{i,j}|
\right)
$$
is an integer, 
the value of $f$
%the maximum value of (\ref{sumxy}) 
is at most $24$
if $|X_{i,j}|$ and $|Y_{i,j}|$ are non-negative integers.
Thus, for $k_3(i,j) = 0,1,\ldots,12$,
the value of $f$
%the maximum value of (\ref{sumxy}) 
is at most $
24 - \frac{13}{7} k_3(i,j)
$
if $|X_{i,j}|$ and $|Y_{i,j}|$ are non-negative integers.
Thus, by the same argument in 
the proof of Lemma \ref{keylemma},
$
\varepsilon({G})- \varepsilon({K_{14}})
$
is at most
$$
-24 |E(H)|
+
3k_3(H)
+
24
|E(H)|
-
\frac{3 \cdot 13}{7}
k_3(H)
-
\frac{a(H)}{4}
=
-\frac{18}{7}
k_3(H)
-
\frac{a(H)}{4}
\leq 0.
$$
Therefore, $\varepsilon({G}) \leq \varepsilon({K_{14}})$.

Suppose that $\varepsilon({G}) = \varepsilon({K_{14}})$.
Then, 
$
-\frac{18}{7}
k_3(H)
-
\frac{a(H)}{4}
\geq 0.
$
Since $k_3(H), a(H) \geq 0$,
we have $k_3(H)=a(H)=0$.
Moreover,
$$
\frac{
|X_{i,j}|+|Y_{i,j}| 
}{2}
+
\frac{1}{4}
|X_{i,j}||Y_{i,j}| + 
\frac{1}{2}
|Y_{i,j}||Z_{i,j}| + 
\frac{1}{2}
|Z_{i,j}||X_{i,j}|
=24
$$
and $|X_{i,j}|+|Y_{i,j}|+|Z_{i,j}| =12$ 
for an arbitrary edge $\{i,j\}$ of $H$.
It is easy to see that $|X_{i,j}|+|Y_{i,j}| \in \{ 7,8\}$.
It then follows that, for an arbitrary $\{i,j\} \in E(H)$,
$(|X_{i,j}|, |Y_{i,j}|,|Z_{i,j}|) \in \{(3,4,5), (4,3,5), (4,4,4)\}$. 
In particular,
the degree of each vertex is either $0$, $4$ or $5$.
%Here,
%for each induced subgraphs (a), (b) or (c),
%the value of weight is, $5/4$, $1$ or $1$ respectively.
%By the argument of the proof of Lemma \ref{keylemma},
%%
%if we have an induced subgraph of type (a),
%the above value is greater than the number of 
%induced subgraphs (a), (b) and (c). 
%Thus, $H$ cannot have any path of length 3 as an induced subgraph.
Moreover, since $k_3(H)=0$,
$\{j\} \cup X_{i,j}$ and $\{i\} \cup Y_{i,j}$ 
are independent sets.
Hence, by $a(H)=0$,
the induced subgraph of $H$ on 
$\{i,j\} \cup X_{i,j} \cup Y_{i,j}$
is the complete bipartite graph
$K_{|X_{i,j}|+1,|Y_{i,j}|+1}$.

Suppose that an edge $\{i,j\}$ of $H$ satisfies
$(|X_{i,j}|, |Y_{i,j}|,|Z_{i,j}|) = (4,4,4)$.
Then, the induced subgraph of $H$ on
$\{i,j\} \cup X_{i,j} \cup Y_{i,j}$
is $K_{5,5}$.
Since the degree of any vertex of $H$ is either, $0$, $4$ or $5$, 
other four vertices are isolated.
Therefore, $G = K_{14}-K_{5,5}$.

It is enough to consider the case that
 $(|X_{s,t}|, |Y_{s,t}|,|Z_{s,t}|) \neq (4,4,4)$ holds
for every edge $\{s,t\}$.
Suppose that $(|X_{i,j}|, |Y_{i,j}|) = (3,4)$.
Then, the induced subgraph of $H$ on 
$\{i,j\} \cup X_{i,j} \cup Y_{i,j}$
is $K_{4,5}$.
Since  $(|X_{s,t}|, |Y_{s,t}|,|Z_{s,t}|)\neq (4,4,4)$ 
for each edge $\{s,t\}$, 
the degree of every vertex in 
$\{i\} \cup Y_{i,j}$
is $4$.
In this case, $K_{4,5}$ is a connected component of $H$. 
Since the degree of other five vertices is at most 4, 
it follows that they are isolated vertices. 
Therefore, $G = K_{14}-K_{4,5}$.

(c)
Let $d \geq 15$
and
let $G = K_d - K_{m,n} \in \Omega_d$.
By Proposition \ref{completebipartite}, 
we have
 \[\varepsilon(G) -\varepsilon(K_d) =\frac{1}{2}mn (m+n-6)
d
-
\frac{1}{4} mn(3mn+2m^2+2n^2-5m-5n-13).\]
When $m=n=5$, we obtain
$\varepsilon(G) -\varepsilon(K_d) = 50(d-14) >0$
as desired.
\end{proof}

%% file: Section2new.tex
\def\daitai{
{ \ 
=\hspace{-1em}\raisebox{1.1ex}{.}\hspace{.1em}\raisebox{-0.2ex}{.}}
\ }

\section{Asymptotic behavior of $\mu_d$}

For $0< p < 1$ and an integer $d >0$, let $G(d, p)$ denote the 
random graph on the vertex set $[d]$ in which
the edges are chosen independently with probability $p$.
For a graph $H$ on the vertex set $[d]$ and
$0< p < 1$, let $G(H, p)$ denote the 
random graph on the vertex set $[d]$ in which
the edges of $H$ are chosen independently with probability $p$
and the edges not belonging to $H$ are not chosen.
Tran--Ziegler \cite{Ziegler} showed that,
for the random graph $G(d, 1/\sqrt{3})$,
$$
\varepsilon (G(d, 1/\sqrt{3})) = \frac{1}{54} d^4 
+ \frac{1}{18} d^3
- \frac{8 }{27} d^2
+ \frac{1}{3} d,
$$
and hence this is a lower bound for $\mu_{d}$.

First, for $d \gg 0$, we give an example of a (non-random) graph $G$
on the vertex set $[d]$
such that
$\varepsilon (G) >\varepsilon(G(d, 1/\sqrt{3}))$.

\begin{Example}
\label{twocompletebipartite}
Let $G = K_d - K_{ad,ad} - K_{(1/2-a)d,(1/2-a)d}$ where 
$a = \frac{1}{28}(7 + \sqrt{21})$
and $d \gg 0$.
%$0\leq a\leq \frac{1}{2}$.
By Propositions \ref{ConnectedComponent} and
\ref{completebipartite},
it follows that
%$
%\varepsilon(G) =
% \left( -\frac{7}{2} a^4+ \frac{7}{2} a^3
%-\frac{9}{8} a^2 + \frac{1}{8} a + \frac{1}{64} \right)d^4 + M(d)
%$
%\label{lowerbound}
%where $M(d)$ is a polynomial in $d$ of degree 3.
%Then, the maximum value of the function
%$g(a) = -\frac{7}{2} a^4+ \frac{7}{2} a^3
%-\frac{9}{8} a^2 + \frac{1}{8} a + \frac{1}{64}$
%($0\leq a\leq \frac{1}{2}$) is
%$\frac{9}{448}$ when $a$ is either
%$\frac{1}{28}(7 - \sqrt{21})$ or $\frac{1}{28}(7 + \sqrt{21})$.
%Thus, for $d \gg 0$, $\varepsilon(G)$ is approximately
$$
\varepsilon(G) =
\frac{9}{448} d^4 
+\frac{1}{7} d^3 
-\frac{103}{112} d^2
+d.
$$
Since
$
1/54 \daitai 0.0185
$
and
$
9/448 \daitai 0.0201
$, we have
$\varepsilon (G) >\varepsilon(G(d, 1/\sqrt{3}))$
for $d \gg 0$.
\end{Example}

Second, we give a random graph ${\mathbb G}$
on the vertex set $[d]$ such that
$\varepsilon ({\mathbb G}) >\varepsilon(G(d, 1/\sqrt{3}))$
for $d \gg 0$.

\begin{Theorem}
\label{randombipartite}
For an integer $d$,
let ${\mathbb G}$  be a random graph
 $K_d  - G(K_{d/2,d/2}, p)$
with $p = 3-\sqrt{5}$.
Then, 
$$
\varepsilon ({\mathbb G}) = \frac{5\sqrt{5}-11}{8} \  d^4 
- \frac{12 \sqrt{5} - 27}{2} d^3
+\frac{19 \sqrt{5} - 44}{2} d^2
+d.
$$
In particular, we have $\varepsilon ({\mathbb G}) >\varepsilon(G(d, 1/\sqrt{3}))$ for all $d \gg 0$.
\end{Theorem}

\begin{proof}
Let $m = d/2$
and let $[d]= V_1 \cup V_2$ be a partition of the vertex set
of $K_{m,m}$.
The number of pairs of edges
$\{i,j\}, \{i,k \}$ 
 satisfying Lemma 1.1 (i) is 
$$
\eta_1=
m(m-1)(m-2)
+
2 m^2 (m-1) (1-p)
+
m^2(m-1) (1-p)^2
$$
where each term corresponds to the case when
(i)
$i,j,k \in V_s$,
(ii)
$i,j\in V_s$, $k \notin V_s $
and
(iii)
$i \in V_s$, $ j,k \notin V_s$,
respectively.
%Since this is a polynomial in the variable $d$ of degree 3,
%we can ignore this part.

Next, we study the number of pairs of edges 
$\{i,j\}, \{k,\ell \}$
satisfying Lemma 1.1 (ii).
Let ${\mathbb G}_{ijk\ell}$ denote the induced subgraph
of ${\mathbb G}$ on the vertex set $\{i,j,k,\ell\} \subset [d]$.
If either ``$i,j,k,\ell \in V_s$" or ``$i,\ell \in V_s$ and $j, k \notin V_s$"
holds, then
$\{i,j,k,\ell\}$ is a cycle of ${\mathbb G}_{ijk\ell}$
whenever $\{i,j\}, \{k,\ell \} \in E({\mathbb G})$.
Hence, we consider the following two cases:
\begin{itemize}
\item[Case 1.]
Suppose $i,j \in V_s$ and $k,\ell \notin V_s$.
Then, ${\mathbb G}_{ijk\ell}$ has 
a cycle of length 4 if and only if 
either 
$\{i,k\}, \{j,\ell \} \in E({\mathbb G})$
or 
$\{i,\ell \}, \{j,k \} \in E({\mathbb G})$
holds.
Thus,
the expected number of pairs of edges 
is
$\eta_2=\binom{m}{2}^2 (1- (1-p)^2)^2 $.
\item[Case 2.]
Suppose that $i \in V_s$ and $j, k,\ell \notin V_s$ hold.
Then, all of $\{k,\ell\}$, $\{ j, k\}$ and $\{j, \ell\}$ 
are edges of ${\mathbb G}$.
On the other hand, $\{i,j\}$ is an edge of ${\mathbb G}$
with probability $1-p$.
If $\{i,j\}$ is an edge of ${\mathbb G}$, 
then ${\mathbb G}_{ijk\ell}$ has 
a cycle of length 4 if and only if 
either 
$\{i,k\} \in E({\mathbb G})$
or 
$\{i,\ell \} \in E({\mathbb G})$
holds.
Thus,
the expected number of pairs of edges 
is
$\eta_3=
m^2 (m-1) (m-2)
(1-p) p^2$.
\end{itemize} 
Therefore,
$
\varepsilon ({\mathbb G}) = 
%\binom{m}{2}^2 (1- (1-p)^2)^2 
%+
%m^2 (m-1) (m-2)
%(1-p) p^2 + 
%\mbox{lower terms}.
\eta_1 + \eta_2 + \eta_3
$.
If $m = d/2$ and $p = 3-\sqrt{5}$,
then
$$
\varepsilon ({\mathbb G}) = \frac{5\sqrt{5}-11}{8} \  d^4 
- \frac{12 \sqrt{5} - 27}{2} d^3
+\frac{19 \sqrt{5} - 44}{2} d^2
+d,
$$
%This is a polynomial in the variable $d$ of degree $4$
%and the coefficient of $d^4$ is
%$$
%\frac{(1- (1-p)^2)^2}{64} +\frac{(1-p) p^2}{16}
%=
%\frac{p^2(p^2-8p+8)}{64}.
%$$
%If 
%$p = 3-\sqrt{5}$,
%then 
whose leading coefficient is $\frac{5\sqrt{5}-11}{8}
\daitai 0.0225425$.
\end{proof}

\begin{Remark}
By Theorem \ref{randombipartite},
the graph $G$ in Example \ref{twocompletebipartite}
does not satisfy $\mu_d = \varepsilon(G)$
for $d\gg 0$.
In fact, for $d=20$,
by Propositions \ref{ConnectedComponent} and
\ref{completebipartite},
it follows that
$$
\max \left\{ \varepsilon(G) \, : \, 
\begin{array}{c}
G \in \Omega_{20}
\mbox{ and each non-empty connected}
\\
\mbox{component of } \overline{G} 
\mbox{ is a complete bipartite graph}
\end{array}
\right\}
= 4176.
$$
Let $G' \in \Omega_{20}$ be the graph
such that $\overline{G'}$ is the bipartite graph 
with $E(\overline{G'} ) =$ 
\begin{center}
{\small
$\{
\{1,12\},\{1,14\},\{1,15\},\{1,16\},\{1,18\},\{1,19\},\{1,20\},
\{2,11\},\{2,12\},\{2,13\},\{2,15\}, $\\
$\{2,17\},\{2,19\},\{2,20\},
\{3,11\},\{3,12\},\{3,13\},\{3,14\},\{3,15\},\{3,16\},\{3,18\},
\{4,14\},$\\
$\{4,15\},\{4,16\},\{4,17\},\{4,18\},\{4,19\},\{4,20\},
\{5,11\},\{5,12\},\{5,13\},\{5,15\},
\{5,17\},$\\
$\{5,18\},\{5,20\},
\{6,12\},\{6,16\},\{6,17\},\{6,18\},\{6,19\},\{6,20\},\{7,11\},\{7,12\},
\{7,13\},$\\
$\{7,14\},\{7,16\},\{7,17\},\{7,19\},
\{8,11\},\{8,12\},\{8,13\},\{8,14\},\{8,15\},\{8,18\},
\{8,19\},$\\
$\{8,20\},
\{9,11\},\{9,14\},\{9,15\},\{9,16\},\{9,17\},\{9,18\},\{9,19\},
\{10,11\},
\{10,13\},\{10,15\},$\\
$\{10,16\},\{10,18\},\{10,19\},\{10,20\}
\}.$
}
\end{center}
Then, $\varepsilon(G') = 4203 > 4176$.
\end{Remark}

\bigskip
\noindent{\bf Acknowledgment.} 
The authors are grateful to an anonymous referee for 
useful suggestions, and helpful comments.

\bigskip